\documentclass[12pt]{article}
 \usepackage{amsthm,amsmath,amssymb, pgf, tikz, pgfpict2e}

\theoremstyle{plain}
\newtheorem{Thm}{Theorem}

\newtheorem{Def}[Thm]{Definition}

\newtheorem{remark}[Thm]{Remark}

\date{June 9, 2017 }
\title{Finding Balance:   Split Graphs and Related Classes}

\author{Karen L. Collins\\
\small Dept. of Mathematics and Computer Science\\
\small Wesleyan University\\
\small Middletown CT 06459-0128\\
\small\tt kcollins@wesleyan.edu\
\and
Ann N. Trenk\thanks{
This work was supported by a grant from the Simons Foundation (\#426725, Ann Trenk). } \\
\small Department of Mathematics\\
\small Wellesley College\\
\small Wellesley MA 02481\\
\small\tt atrenk@wellesley.edu
}

\begin{document}
\maketitle

\begin{abstract} 
A graph is a split graph if its vertex set can be partitioned into a clique and a stable set.  A split graph is unbalanced   if there exist two such partitions that are distinct.   Cheng, Collins and Trenk (2016), discovered the following interesting counting fact:   unlabeled, unbalanced split graphs  on $n$ vertices can be placed into a bijection with all unlabeled split graphs on $n-1$ or fewer vertices. In this paper we  translate these concepts and the theorem to different combinatorial settings: minimal set covers, bipartite graphs with a distinguished block and posets of height one. 

\end{abstract}

\bibliographystyle{plain} 

 \bigskip\noindent \textbf{Keywords:  split graph, set cover, bipartite graph, bipartite poset, bijection}

\bigskip
\noindent

\section{Introduction} 
\subsection{Background}
In this paper, we consider unlabeled graphs as in \cite{We01}, that is, two graphs are considered the same if there is an isomorphism between them.  For any graph $G$, the number of vertices in a largest clique is denoted by $\omega(G)$ and the number of vertices in a largest stable (independent) set is denoted $\alpha(G)$. We denote by $G[X]$, the graph induced in $G$ by $X\subseteq V(G)$. A finite graph $G$ is a \emph{split graph} if its vertex set can be partitioned into $K\cup S$ where $G[K]$ is a clique and $G[S]$ is a stable set.  We refer to such a partition as a \emph{split graph partition} or a $KS$-partition of $G$.

It is easy to see from the definition that the complement of a split graph is again a split graph, and that a split graph does not contain a chordless odd cycle on five or more vertices. Therefore by the Strong Perfect Graph Theorem, split graphs are perfect. Indeed, the closely related class of double split graphs are one of the five families of graphs that form the base case of the  inductive proof of the Strong Perfect Graph Theorem, conjectured in \cite{Be61} and proven in  \cite{ChRoSeTh06}. 

In \cite{ChCoTr16}, we categorized split graphs based on their $KS$-partitions and defined \emph{balanced} and \emph{unbalanced} split graphs. The terms  \emph{balanced} and \emph{unbalanced} in Definition~\ref{bal-unbal-def} refer to a split graph $G$ while the terms \emph{$K$-max} and \emph{$S$-max} refer to a particular $KS$-partition of $G$.

 \begin{Def} {\rm
 A split graph $G$  is \emph{balanced} if it has a $KS$-partition satisfying $|K| = \omega (G)$ and $|S| = \alpha(G)$   and \emph{unbalanced} otherwise.     A $KS$-partition is \emph{$S$-max} if $|S| = \alpha(G)$    and \emph{$K$-max} if $|K| = \omega (G)$.     }
 \label{bal-unbal-def}
 \end{Def}

The first and last columns of Figure~\ref{big-fig}  show all nine split graphs on four vertices.  The first column shows a $K$-max partition of the vertices and the last column shows an $S$-max partition.   The graphs in the first eight rows are unbalanced and the graph  in the ninth row is balanced.   As this example illustrates,
 the $KS$-partitions of an unbalanced   split graph are not   unique. The next theorem  follows from the work of Hammer and Simeone \cite{HaSi81} and appears 
in~\cite{Go80}.  
  
\bigskip
  
\begin{Thm} {\rm (Hammer and Simeone)}
 For any $KS$-partition of a split graph $G$, exactly one of the following holds:
 
 (i)  $|K| = \omega(G)$ and $|S| = \alpha(G)$.  \hfill (balanced)
 
 (ii)  $|K| = \omega(G)-1$ and $|S| = \alpha(G)$. \hfill (unbalanced, $S$-max)

(iii)  $|K| = \omega(G)$ and $|S| = \alpha(G)-1$. \hfill (unbalanced, $K$-max)

\smallskip

Moreover, in (ii) there exists $s \in S$ so that $K \cup \{s\}$ is complete and in  (iii) there exists $k \in K$ so that $S \cup \{k\}$ is a stable set.

\label{split-thm}
 \end{Thm}
 
In cases (ii) and (iii) of Theorem~\ref{split-thm}, we call such  vertices $s$ or $k$  \emph{swing} vertices of $G$.  Additional background on split graphs can be found in  \cite{ChCoTr16} and \cite{Go80}. The next remarks follow directly from Theorem~\ref{split-thm} and give alternative conditions for a split graph to be balanced or unbalanced.

\begin{remark} \label{swing-rem}
{\rm
A split graph is unbalanced if and only if it has a swing vertex. }

\end{remark}
 
\begin{remark} {\rm
Since we are considering unlabeled graphs, an unbalanced split graph will have a unique $K$-max partition and a unique $S$-max partition.   A balanced split graph has a unique partition that is both $K$-max and $S$-max.  }
\end{remark}
 
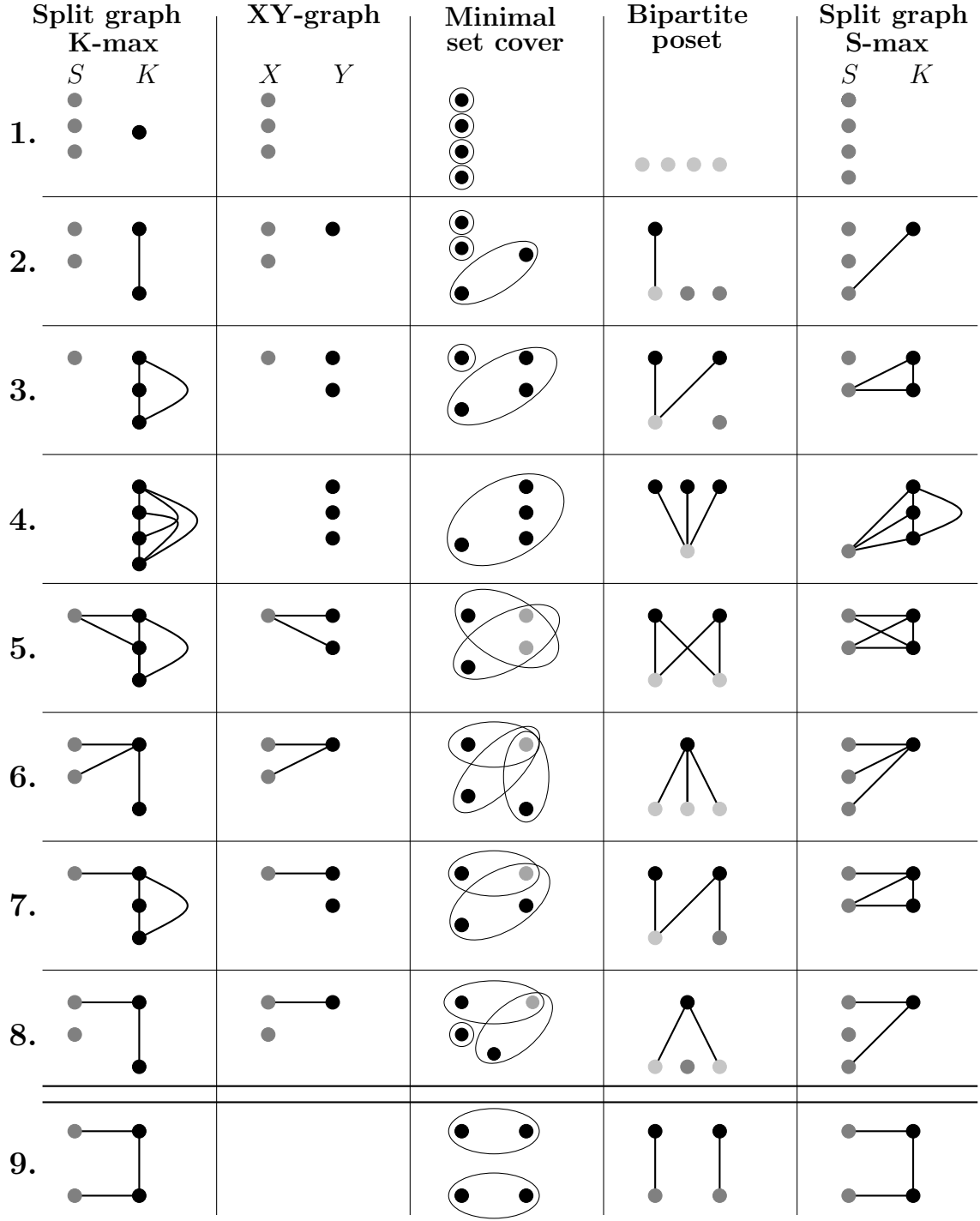
\begin{figure}
 
\begin{center}
\begin{tikzpicture}

\node(C1) at (-.8,.5) {\bf \large 9.};
\node(C1) at (-.8,2.5) {\bf \large 8.};
\node(C1) at (-.8,4.5) {\bf \large 7.};
\node(C1) at (-.8,6.5) {\bf \large 6.};
\node(C1) at (-.8,8.5) {\bf \large 5.};
\node(C1) at (-.8,10.5) {\bf \large 4.};
\node(C1) at (-.8,12.5) {\bf \large 3.};
\node(C1) at (-.8,14.5) {\bf \large 2.};
\node(C1) at (-.8,16.5) {\bf \large 1.};

%\draw (-.4,-.3)--(-.4,18.5);
\draw (2.2,-.3)--(2.2,18.3);
\draw (5.2,-.3)--(5.2,18.3);
\draw (8.2,-.3)--(8.2,18.3);
\draw (11.2,-.3)--(11.2,18.3);

%\draw(0,1.4)--(14,1.4);

\draw[thick] (-.5,1.7)--(14,1.7);
\draw[thick] (-.5,1.45)--(14,1.45);
\draw (-.5,3.5)--(14,3.5);
\draw (-.5,5.5)--(14,5.5);
\draw (-.5,7.5)--(14,7.5);
\draw (-.5,9.5)--(14,9.5);
\draw (-.5,11.5)--(14,11.5);
\draw (-.5,13.5)--(14,13.5);
\draw (-.5,15.5)--(14,15.5);
%\draw (-.5,17.4)--(14,17.4);

\draw[thick] (0,0) -- (1,0);
\draw[thick] (0,1) --(1,1) -- (1,0);
\draw[thick] (0,3) -- (1,3) -- (1,2);
\draw[thick] (0,5) -- (1,5) -- (1,5) --(1,4);
\draw[thick] (1,4) .. controls (2,4.5) and (2,4.5) .. (1,5);
\draw[thick] (0,6.5) -- (1,7) -- (1,6);
\draw[thick] (0,7) -- (1,7);
\draw[thick] (1,9) -- (0,9) -- (1,8.5) -- (1,8) --(1,9);
\draw[thick] (1,8) .. controls (2,8.5) and (2,8.5) .. (1,9);
\draw[thick] (1,9.8)--(1,11);
\draw[thick] (1,9.8) .. controls (2.2,10.5) and (2.2,10.5) .. (1,11);
\draw[thick] (1,9.8) .. controls (1.8,10.5) and (1.8,10.5) .. (1,10.6);
\draw[thick] (1,10.2) .. controls (1.8,10.5) and (1.8,10.5) .. (1,11);
\draw[thick] (1,13) -- (1,12);
\draw[thick] (1,12) .. controls (2,12.5) and (2,12.5) .. (1,13);
\draw[thick] (1,14) -- (1,15);

\filldraw [gray] 
(0,0) circle [radius=3pt]
(0,1) circle [radius=3pt]
(0,2.5) circle [radius=3pt]
(0,3) circle [radius=3pt]
(0,5) circle [radius=3pt]
(0,6.5) circle [radius=3pt]
(0,7) circle [radius=3pt]
(0,9) circle [radius=3pt]
(0,13) circle [radius=3pt]
(0,14.5) circle [radius=3pt]
(0,15) circle [radius=3pt]
(0,16.2) circle [radius=3pt]
(0,16.6) circle [radius=3pt]
(0,17) circle [radius=3pt]
;

\filldraw [black] 
(1,0) circle [radius=3pt]
(1,1) circle [radius=3pt]
(1,2) circle [radius=3pt]
(1,3) circle [radius=3pt]
(1,4) circle [radius=3pt]
(1,4.5) circle [radius=3pt]
(1,5) circle [radius=3pt]
(1,6) circle [radius=3pt]
(1,7) circle [radius=3pt]
(1,8) circle [radius=3pt]
(1,8.5) circle [radius=3pt]
(1,9) circle [radius=3pt]
(1,9.8) circle [radius=3pt]
(1,10.2) circle [radius=3pt]
(1,10.6) circle [radius=3pt]
(1,11) circle [radius=3pt]
(1,12) circle [radius=3pt]
(1,12.5) circle [radius=3pt]
(1,13) circle [radius=3pt]
(1,14) circle [radius=3pt]
(1,15) circle [radius=3pt]

(1,16.5) circle [radius=3pt]
;
\node (A0) at (.6,17.4) {\bf $S$ \ \ \ \ $K$};
\node (A) at (.6,17.9) {\bf $\mathbf K$-max};
\node (AA) at (.5,18.3) {\bf Split graph};

\draw[thick] (3,3) -- (4,3);
\draw[thick] (3,5) -- (4,5);
\draw[thick] (3,6.5) -- (4,7)--(3,7);
\draw[thick] (4,9) -- (3,9) -- (4,8.5);

\filldraw [gray] 
(3,2.5) circle [radius=3pt]
(3,3) circle [radius=3pt]
(3,5) circle [radius=3pt]
(3,6.5) circle [radius=3pt]
(3,7) circle [radius=3pt]
(3,9) circle [radius=3pt]
(3,13) circle [radius=3pt]
(3,14.5) circle [radius=3pt]
(3,15) circle [radius=3pt]
(3,16.2) circle [radius=3pt]
(3,16.6) circle [radius=3pt]
(3,17) circle [radius=3pt]
;

\filldraw [black] 

(4,3) circle [radius=3pt]
(4,4.5) circle [radius=3pt]
(4,5) circle [radius=3pt]
(4,7) circle [radius=3pt]
(4,8.5) circle [radius=3pt]
(4,9) circle [radius=3pt]
(4,10.2) circle [radius=3pt]
(4,10.6) circle [radius=3pt]
(4,11) circle [radius=3pt]
(4,12.5) circle [radius=3pt]
(4,13) circle [radius=3pt]
(4,15) circle [radius=3pt]
;

\node (B0) at (3.6, 17.4) {\bf $X$ \ \ \ \ $Y$ };
\node (B) at (3.5,18.3) {\ \ \ \bf $\mathbf X\mathbf Y$-graph};

\draw (6.5,0) ellipse [x radius=20pt, y radius=10pt];
\draw (6.5,1) ellipse [x radius=20pt, y radius=10pt];

\draw (6.5,3) ellipse [x radius=22pt, y radius=9.5pt];
\draw[rotate=40.95] (6.85,-2.5) ellipse [x radius=20.5pt, y radius=10.5pt];
\draw (6,2.5) circle [radius=5.3pt];

\draw (6.5,5) ellipse [x radius=20pt, y radius=10pt];
\draw[rotate=31.1] (8,.5) ellipse [x radius=24.5pt, y radius=13pt];

\draw (6.5,7) ellipse [x radius=20pt, y radius=10pt];
\draw (7,6.5) ellipse [x radius=10pt, y radius=20pt];
\draw[rotate=45] (9.3,.05) ellipse [x radius=25pt, y radius=10pt];

\draw[rotate=29.4] (10.05,4.2) ellipse [x radius=26pt, y radius=12pt];
\draw[rotate=60.5] (10.96,-1.5) ellipse [x radius=14pt, y radius=25pt];

\draw[rotate=29.7] (11,5.8) ellipse [x radius=28pt, y radius=17pt];

\draw[rotate=30.3] (12.06,7.5) ellipse [x radius=27pt, y radius=12pt];
\draw (6,13) circle [radius=6pt];

\draw (6,15.1) circle [radius=5.3pt];
\draw (6,14.7) circle [radius=5.3pt];
\draw[rotate=32] (13.1,8.7) ellipse [x radius=22pt, y radius=9pt];

\draw (6,17) circle [radius=5.3pt];
\draw (6,16.6) circle [radius=5.3pt];
\draw (6,16.2) circle [radius=5.3pt];
\draw (6,15.8) circle [radius=5.3pt];

\filldraw [black] 
(6,0) circle [radius=3pt]
(6,1) circle [radius=3pt]
(6,2.5 ) circle [radius=2.8pt]
(6,3) circle [radius=2.8pt]
(6,5) circle [radius=3pt]
(6.1,6.2) circle [radius=3pt]
(6.1,7) circle [radius=3pt]
(6.1,9) circle [radius=3pt]
(6,13) circle [radius=3pt]
(6,14) circle [radius=3pt]
(6,14.7) circle [radius=2.8pt]
(6,15.1) circle [radius=2.8pt]
(6,15.8) circle [radius=2.8pt]
(6,16.2) circle [radius=2.8pt]
(6,16.6) circle [radius=2.8pt]
(6,17) circle [radius=2.8pt]
%;
%\filldraw [black] 
(7,0) circle [radius=3pt]
(7,1) circle [radius=3pt]
(6.5,2.2) circle [radius=2.8pt]
%(7,3) circle [radius=3pt]
(6,4.2) circle [radius=3pt]
(7,4.5) circle [radius=3pt]
%(7,5) circle [radius=3pt]
(7,6) circle [radius=3pt]
%(7,7) circle [radius=3pt]
(6.1,8.2) circle [radius=3pt]
%(7,8.5) circle [radius=3pt]
%(7,9) circle [radius=3pt]
(6,10.1) circle [radius=3pt]
(7,10.2) circle [radius=3pt]
(7,10.6) circle [radius=3pt]
(7,11) circle [radius=3pt]
(6,12.2) circle [radius=3pt]
(7,12.5) circle [radius=3pt]
(7,13) circle [radius=3pt]
%(7,14) circle [radius=3pt]
(7,14.6) circle [radius=3pt]
%(7,16.5) circle [radius=3pt]
;

\filldraw [gray!70!white]
(7.1,3) circle [radius=2.8pt]
(7,5) circle [radius=3pt]
(7,7) circle [radius=3pt]
(7,8.5) circle [radius=3pt]
(7,9) circle [radius=3pt]

;
\node (C0) at (6.6, 18.3) {\bf Minimal};
\node (C) at (6.6,17.9) {\bf \ set cover};

\draw[thick] (9,0)--(9,1);
\draw[thick] (10,0)--(10,1);
\draw[thick] (9,2)--(9.5,3)--(10,2);
\draw[thick] (9,5)--(9,4)--(10,5)--(10,4);
\draw[thick] (9,6)--(9.5,7)--(9.5,6)--(9.5,7)--(10,6);
\draw[thick] (10,8)--(9,9)--(9,8)--(10,9)--(10,8);
\draw[thick] (9,11)--(9.5,10)--(9.5,11)--(9.5,10)--(10,11);
\draw[thick] (9,13)--(9,12)--(10,13);
\draw[thick] (9,14)--(9,15);

\filldraw[gray!45!white]
(9,2) circle [radius=3pt]
(10,2) circle [radius=3pt]
(9,4) circle [radius=3pt]
(10,4) circle [radius=3pt]
(9,6) circle [radius=3pt]
(9.5,6) circle [radius=3pt]
(10,6) circle [radius=3pt]
(9,8) circle [radius=3pt]
(10,8) circle [radius=3pt]
(9.5,10) circle [radius=3pt]
(9,12) circle [radius=3pt]
(9,14) circle [radius=3pt]

(8.8,16) circle [radius=3pt]
(9.2,16) circle [radius=3pt]
(9.6,16) circle [radius=3pt]
(10,16) circle [radius=3pt]

;
\filldraw[black!50!white]
(9,0) circle [radius=3pt]
(10,0) circle [radius=3pt]
(9.5,2) circle [radius=3pt]
(10,4) circle [radius=3pt]
(10,12) circle [radius=3pt]
(9.5,14) circle [radius=3pt]
(10,14) circle [radius=3pt]

;

\filldraw [black] 
(9,1) circle [radius=3pt]
(10,1) circle [radius=3pt]
(9.5,3) circle [radius=3pt]

(9,5) circle [radius=3pt]
(10,5) circle [radius=3pt]

(9.5,7) circle [radius=3pt]

(9,9) circle [radius=3pt]
(10,9) circle [radius=3pt]

(9,11) circle [radius=3pt]
(9.5,11) circle [radius=3pt]
(10,11) circle [radius=3pt]
(9,13) circle [radius=3pt]
(10,13) circle [radius=3pt]

(9,15) circle [radius=3pt]
;

\node (D) at (9.5,18.3) {\bf  Bipartite};
\node (D1) at (9.5,17.9) {\bf poset};

\draw[thick] (12,0) -- (13,0);
\draw[thick] (12,1) --(13,1) -- (13,0);
\draw[thick] (12,3)--(13,3)--(12,2);
\draw[thick] (12,5)--(13,5)--(12,4.5)--(13,4.5)--(13,5);
\draw[thick] (12,7)--(13,7)--(12,6);
\draw[thick] (12,6.5)--(13,7);
\draw[thick] (13,8.5)--(12,8.5)--(13,9)--(13,8.5)--(12,9)--(13,9);
\draw[thick] (12,10)--(13,11)--(13,10.6)--(12,10)--(13,10.2)--(13,10.6);
\draw[thick] (13,11) .. controls (14,10.6) and (14,10.6) .. (13,10.2);
\draw[thick] (12,12.5)--(13,12.5)--(13,13)--(12,12.5);
\draw[thick] (12,14)--(13,15);

\filldraw [gray] 
(12,0) circle [radius=3pt]
(12,1) circle [radius=3pt]
(12,2) circle [radius=3pt]
(12,2.5) circle [radius=3pt]
(12,3) circle [radius=3pt]
(12,4.5) circle [radius=3pt]
(12,5) circle [radius=3pt]
(12,6) circle [radius=3pt]
(12,6.5) circle [radius=3pt]
(12,7) circle [radius=3pt]
(12,8.5) circle [radius=3pt]
(12,9) circle [radius=3pt]
(12,10) circle [radius=3pt]
(12,12.5) circle [radius=3pt]
(12,13) circle [radius=3pt]
(12,14) circle [radius=3pt]
(12,14.5) circle [radius=3pt]
(12,15) circle [radius=3pt]
(12,15.8) circle [radius=3pt]
(12,16.2) circle [radius=3pt]
(12,16.6) circle [radius=3pt]
(12,17) circle [radius=3pt]
(12,17) circle [radius=3pt]
;

\filldraw [black] 
(13,0) circle [radius=3pt]
(13,1) circle [radius=3pt]
%(13,2) circle [radius=3pt]
(13,3) circle [radius=3pt]
%(13,4) circle [radius=3pt]
(13,4.5) circle [radius=3pt]
(13,5) circle [radius=3pt]
%(13,6) circle [radius=3pt]
(13,7) circle [radius=3pt]
%(13,8) circle [radius=3pt]
(13,8.5) circle [radius=3pt]
(13,9) circle [radius=3pt]
(13,10.2) circle [radius=3pt]
(13,10.6) circle [radius=3pt]
%(13.5,10) circle [radius=3pt]
(13,11) circle [radius=3pt]
%(13,12) circle [radius=3pt]
(13,12.5) circle [radius=3pt]
(13,13) circle [radius=3pt]
%(13,14) circle [radius=3pt]
(13,15) circle [radius=3pt]

;

\node (E0) at (12.6,17.4) {\bf $S$ \ \ \ \ $K$};
\node (E) at (12.6,17.9) {\bf $\mathbf S$-max};
\node (EE) at (12.7,18.3) {\bf Split graph};

\end{tikzpicture}

 \end{center}
 
\caption{\small All split graphs, minimal set covers, and bipartite posets on four vertices.   All $XY$-graphs on 3 vertices.   Entries in rows 1--8 are unbalanced, entries in row 9 are balanced.}
\label{big-fig}
\end{figure}

\subsection{Overview}
In this paper we study bijections and counting questions for split graphs and related classes.  
  There is a complicated formula for the number of  split  graphs on $n$ vertices resulting from the  work of Clarke \cite{Cl90} and Royle \cite{Ro00}. In \cite{ChCoTr16} we count the number of balanced and unbalanced split graphs on $n$ vertices. That proof uses a sequence of bijections involving families of graphs called NG-graphs.   The following  surprising theorem about split graphs is also proven in \cite{ChCoTr16}. In Section~\ref{comp-for-split-g} of this paper, we give a direct and natural proof of Theorem~\ref{compil-split-thm}. 
   
\begin{Thm} (Compilation Theorem for Split Graphs)
There is a  bijection between the class of  unbalanced split graphs on $n$ vertices and the class of  split graphs on $t$ vertices where $0 \le t \le n-1$.
\label{compil-split-thm}
\end{Thm}

Royle's work provides a bijection between split graphs and minimal set covers. In searching the \emph{Online Encyclopedia of Integer Sequences\/} (OEIS), we found on the split graph page (A048194) two additional  classes of combinatorial objects with similar sizes. The first class is bipartite graphs with a distinguished block ($XY$-graphs). %These were studied by Robinson \cite{Ro76}. 
The number of $XY$-graphs on $n$ vertices is equal to the number of principal (or fundamental)  transversal matroids of size $n$, as enumerated by Brylawski \cite{Br75}.  In 2000, Vladeta Jovivic noted on the OEIS page that the number of unlabeled bipartite graphs with $n$ vertices and no isolated vertices in the distinguished bipartite block equals the number of unlabeled split graphs on $n$ vertices. The second class is  unlabeled posets of height at most one. Detlef Pauly noted on the OEIS page that the number of such posets with $n$ elements equals  number of unlabeled split graphs on $n$ vertices. 

\medskip

\addtocounter{table}{1}

\begin{table}[h]
\begin{center}
\begin{tabular}{|c|c|}\hline
 {\bf Class} &  {\bf Unbalanced subclass characterized   } \\ 
 &  {\bf  by the existence of  a  \ldots } \\ 
  \hline
  & \\
 Split Graphs &   swing vertex \\ \hline
   & \\
 Minimal Set Cover $\cal C$ on set $V$ &  set of size $|V| - |{\cal C}| + 1$ \\ \hline
   & \\
 $XY$-graph with no isolate in $Y$ &   universal vertex in $X$ \\ \hline
   & \\
 Bipartite Poset &   full support point\\  \hline
\end{tabular}
\end{center}
\caption{Unbalanced subclasses characterized by the existence of a structure.}
\label{unbal-table}
\end{table}

In this paper we consider these four combinatorial settings. We define balanced and unbalanced subclasses for each of the other three classes. As detailed in Table~\ref{unbal-table},
for each class, the unbalanced category is characterized by the existence of a structure.  
We provide bijections between each pair of classes, and show our bijections preserve balance.  These bijections prove that a Compilation Theorem holds for each of the other three classes. We also prove each of these directly, and each new context contributes something new: either a shorter and more intuitive proof, and/or an intriguing  new theorem. It is interesting to look at these proofs in their own settings without reverting to split graphs. Some of these proofs are subtle and would be more difficult to discover without their connection to one of the other classes. 

Other authors have studied counting questions for related classes of graphs and posets.   Hanlon \cite{Ha79} counted unlabeled bipartite graphs using generating functions.   In 2014, Gainer-Dewar and Gessel   counted  both unlabeled bipartite graphs and blocks in a bipartite graph \cite{GaGe14}.  Also in 2014, Guay-Paquet et al. use  similar techniques to count  both labeled and unlabeled $(3+1)$-free posets \cite{GuMoRo14}. 

The rest of the paper is organized as follows:  we focus on minimal set covers in Section~\ref{set-sec},  $XY$-graphs in Section~\ref{bip-sec}, and bipartite posets in Section~\ref{poset-sec}. In Section~5 we show the remaining bijections between different classes. Finally, in Section~\ref{comp-for-split-g}, we give a short proof of Theorem~\ref{compil-split-thm} and  in Section~\ref{concl} we present concluding remarks.

\section{Minimal set covers}

\label{set-sec}

In 1990, Clarke \cite{Cl90} found counting formulas for labeled and unlabeled set covers and in 2000, Royle \cite{Ro00} recognized that the number of unlabeled minimal set covers equalled the  number of unlabeled split graphs for small values of $n$. He demonstrated a bijection between unlabeled split graphs and unlabeled minimal set covers, thus confirming that Clarke's formula counts split graphs. In this section, we translate the concepts of balanced and unbalanced split graphs to the setting of minimal set covers.   A minimal set cover is unlabeled when the elements in the ground set of a set cover are unlabeled and the sets in the cover are  also unlabeled.  Column 3 of Figure~\ref{big-fig} shows the nine minimal set covers of a set with four elements;  the set cover in row~1  contains four sets while that in row 2 contains 3 sets.

\begin{Def}{\rm
Given a set of unlabeled vertices $V$, a \emph{set cover}  ${\cal C}$ is a collection of subsets of $V$ whose union is $V$.   A set cover $\cal C$ is  \emph{minimal} if no set in $\cal C$ is contained in the union of the remaining sets in $\cal C$.  A  vertex in $V$ is \emph{loyal} if it is in a unique set of $\cal C$.  

}
\label{msc-def}
\end{Def}

%An example showing all  nine minimal set covers on four vertices is given in Figure~\ref{}
In column three of Figure~\ref{big-fig}, the loyal vertices are colored black.
The next remark follows directly from Definition~\ref{msc-def}.

\begin{remark}{\rm
A set cover $\cal C$ of $V$ is minimal if and only if each  set in $\cal C$ contains a loyal vertex.  
}
\label{loyal-rem}
\end{remark}

While the definition of loyal vertices in set covers comes from \cite{Ro00},  we
next define an analog  in split graphs.
 
\begin{Def}{\rm
A vertex in a split graph $G$ is \emph{loyal} if it is contained in a unique maximal clique of $G$.
}
\end{Def}

 \begin{remark} {\rm For any $KS$-partition of a split graph $G$, each vertex in $S$ is loyal. 
 For any  $S$-max $KS$-partition of a split graph $G$, the only loyal vertices in $G$ are those in $S$ unless there is a unique swing  vertex $s\in S$.  If such an $s$ exists,  the additional loyal vertices, if any, are those vertices of $K$ whose only adjacency in $S$ is $s$. 
 }
 \label{loyal-S-rem}
 \end{remark}
 
 In the last column of Figure~\ref{big-fig}, the graphs with loyal vertices in $K$ are those in rows 2, 3, 4, and 7.
 The next theorem appears in  \cite{Ro00}.  We include a proof here because we refer to the bijection in the proof  of Theorem~\ref{Royle-thm}
 when proving Theorem~\ref{set-preserve-bal}.    The construction in the proof of Theorem~\ref{Royle-thm}
  is illustrated in Figure~\ref{big-fig} by comparing the entries in columns  3 and 5.

 \begin{Thm}[Royle \cite{Ro00}]
There is a bijection between split graphs on $n$ vertices   and minimal set covers on  $n$ vertices.
\label{Royle-thm}
 \end{Thm}
 
 \begin{proof}
% REWRITTEN PROOF BY ANN
 
 Let $\cal C$ be a minimal set cover of an  $n$ element set $V$.  Form a graph $G$ on the same vertex set as follows.  Choose one (representative) loyal vertex from each set of $\cal C$,  let $S$ be the collection of these loyal vertices, and let $K = V-S$.    Form the edge set of $G$ as follows:  $uv \in E(G)$ if and only if either $u$ and $v$ are together in a set of $\cal C$ or if $u$,$v$ are both in $K$.  The result is a split graph and $K \cup S$ forms a $KS$-partition.    Indeed, by construction, $|S| = |{\cal C}|$.  Furthermore, for each $k \in K$  there exists $s \in X$ for which $k$ and $s$ are in the same set of   $\cal C$ and thus $ks\in E(G)$.   Consequently $K \cup S$ is an $S$-max $KS$-partition of $G$.
 
 The process is reversible.  Let $G = (V,E)$ be a split graph with an $S$-max   $KS$-partition.  For each $s \in S$, form a set consisting of $s$ and its neighbors.  The resulting collection of sets is a set cover $\cal C$ of $V$ since each $k\in K$ is adjacent to some $s \in S$.  The vertices of $S$ are loyal in $\cal C$ by Remark~\ref{loyal-S-rem},  hence by Remark~\ref{loyal-rem},
  $\cal C$ is a minimal set cover.  %We remark that the $KS$-partition of $G$ resulting from this bipartition is $S$-max.
\end{proof}

 We next define  analogs of balanced and unbalanced for minimal set covers so that the bijection given in  Theorem~\ref{Royle-thm}  preserves balance.   The unbalanced class is defined by the existence of a set of a particular size.   It is easy to check that the split graphs and minimal set covers in rows 1--8 of Figure~\ref{big-fig} are unbalanced while the ones in row 9 are balanced.

\begin{Def} {\rm
Let $\cal C$ be a minimal set cover of $V$.  Then $\cal C$ is \emph{unbalanced} if it contains a set with cardinality $|V| - |{\cal C}| + 1$ and \emph{balanced} otherwise.  
}
\end{Def}

\begin{Thm}
The bijection given in Theorem~\ref{Royle-thm} maps unbalanced minimal set covers to unbalanced split graphs and balanced minimal set covers to balanced split graphs.
\label{set-preserve-bal}
\end{Thm}

\begin{proof}
Let $\cal C$ be an unbalanced minimal set cover of $V$.  Choose a set $Y \in {\cal C}$ with $|Y| =  |V| - |{\cal C}| + 1$.  Let $V'$ be the vertices in $V $ that are not in $Y$  and let ${\cal C}'$ be the collection of sets in $\cal C$ other than $Y$.    Hence,
   $|V'| = |V| - (|V| - |{\cal C}| + 1) = |{\cal C}| - 1 = |{\cal C}'|$.   By Remark~\ref{loyal-rem}, each set in ${\cal C}'$ has a loyal vertex and by Definition~\ref{msc-def}, these loyal vertices cannot be in $Y$.  Thus they all come from $V'$ and therefore %$|V'| \ge |{\cal C}'|$.   Hence 
   each vertex in $V'$ belongs to a unique set in ${\cal C}'$ and each set in ${\cal C}'$ consists of  one vertex from $V'$ and a (perhaps empty) subset of $Y$.  

Let $u$ be a loyal vertex chosen  from $Y$, and use the bijection from Theorem~\ref{Royle-thm} to get a split graph $G$.   The set of loyal vertices chosen from the sets in ${\cal C}'$ will be $V'$.   In the resulting $KS$-partition of $G$, the vertex set of   clique $K$  is $Y - \{u\}$.      Adding vertex $u$ to $K$ gives a larger clique since the bijection transforms the set $Y$ into a clique.  Thus the split  graph $G$ is unbalanced as desired. 

Conversely, let $G$ be an unbalanced split graph  with vertex set $V$.  Fix an $S$-max $KS$-partition of $G$.  The bijection in Theorem~\ref{Royle-thm} produces a set cover $\cal C$ of $V$ where each set in $\cal C$ corresponds to an element of $S$ together with its neighborhood.  Thus $|{\cal C}| = |S|$.  Since $G$ is unbalanced and our partition is $S$-max, there exists a vertex in $S$ adjacent to every vertex in $K$.   Let $u$ be such a vertex, thus $K \cup \{u\}$ is a maximal clique in $G$  and its vertices form a set $Y$  in $\cal C$.    Thus in the resulting set cover $\cal C$, we have    $|V| - |{\cal C}| + 1 = (|K| + |S|) - |S| + 1 = |K| + 1$.  The set $Y$ in $\cal C$ has cardinality $|Y| = |K \cup \{u\}| = |K| + 1$, so $\cal C$ is unbalanced as desired.
\end{proof}

The next theorem is the Compilation Theorem for Minimal Set Covers.  %The bijection given in   Theorem~\ref{compil-set-thm} is more natural than the bijection between the corresponding classes of split graphs given in \cite{ChCoTr16}. 
Our proof focuses on a set $Y$ whose existence is the defining property of unbalanced minimal set covers.

\begin{Thm}
There is a bijection between unbalanced minimal set covers of an $n$-set and minimal set covers of a set with $t$   elements for $0 \le t \le n-1$.
\label{compil-set-thm} 
\end{Thm}
\begin{proof}
Let $V$ be an $n$-set and ${\cal C}$ an unbalanced minimal set cover of $V$.  We show how to map ${\cal C}$ to a minimal set cover ${\cal C'}$ on a set with $n-1$ or fewer elements.    Let $k = |{\cal C}|$.  Since set cover ${\cal C}$ is unbalanced, there exists a set $Y$ in ${\cal C}$ with $|Y| = n-k+1.$  Let $V'$ be the set of elements of $V$ covered by $Y$ and not covered by any other set in ${\cal C}$.  We know $|V'| \ge 1$ since $Y$ has a loyal vertex, so $|V-V'| \le n-1$.  Define ${\cal C'}$ to be the set cover of $V-V'$ consisting of all the sets in ${\cal C}$ except $Y$.  Then $|{\cal C'}| = |{\cal C}| -1$.  Since ${\cal C}$ is a minimal set cover, each set $S$ in ${\cal C}$ contains a loyal vertex which is still loyal to $S$ in ${\cal C'}$.    Thus ${\cal C'}$ is a \emph{minimal} set cover of $V-V'$ and $|V-V'| \le n-1$. 

Conversely, let ${\cal C'}$ be a minimal set cover of a set $V'$ with $t$ elements, where $0 \le t \le n-1$.  Let $k' = |{\cal C'}|$.  For each set in ${\cal C'}$, designate a loyal element.  Create a new set $V$ consisting of $V'$ and $n-t$ additional elements, and a set cover ${\cal C}$ consisting of the sets in ${\cal C'}$ together with a set $Y$ that contains the $n-t$ new elements together with the elements of $V'$ other than the $k'$ loyal elements designated earlier.  It follows immediately that ${\cal C}$ is a minimal set cover of $n$-set $V$ and $|{\cal C}| = k' + 1$.    

We next show that ${\cal C}$ is unbalanced.    Note that $Y$ consists of the $n-t$ new elements together with $t-k'$ elements of $V'$, thus $|Y| = n-k'$.  Furthermore, $|{\cal C}| = |{\cal C'}| + 1 = k' + 1$ and $|V| - |{\cal C}|  + 1 = n - (k' + 1) + 1 = n-k' = |Y|$, so ${\cal C}$ is unbalanced.  

Finally, the second mapping reverses the first so we have a bijection.
\end{proof}

\section{ $XY$-Graphs}
\label{bip-sec}

In \cite{ChCoTr16}  we calculate the number of (unlabeled) unbalanced split graphs on $n$ vertices for $n = 1,2,3, \ldots$ and get the following sequence:  $1,2,4,8,17,38,94,258  \ldots$.     According to the  Online Encyclopedia of Integer Sequences (OEIS), this sequence also counts the number of bipartite graphs with a distinguished block.  This motivated us to seek connections between these two classes, and indeed the  proof of Theorem~\ref{OEIS-thm}   gives an explicit bijection between these classes.   We prefer to use the term \emph{$XY$-graph} rather than \emph{bipartite graph with a distinguished block} so that we may consistently 
refer to $X$ as the distinguished block and $Y$ as the other block.

\begin{Def} {\rm An \emph{$XY$-graph} is a  bipartite graph together with a bipartition of the vertices as $X \cup Y$.    The set $X$ is the \emph{distinguished part} of the bipartition. }
\end{Def}

While there are only  three distinct bipartite graphs on three vertices, there are eight  different $XY$-graphs on three vertices, as shown in 
Column 2 of Figure~\ref{big-fig}.  For example, the bipartite graph $P_3$ is counted as two different $XY$-graphs, one with $|X| = 1$ and one with $|X|=2$.    These eight $XY$-graphs 
  correspond to the unbalanced split graphs on four vertices, as described below in Theorem~\ref{OEIS-thm} and illustrated in  the first two columns of  Figure~\ref{big-fig}.

\begin{Thm}
There  is a bijection  between the class  of $XY$-graphs on $n$ vertices 
and the class of unbalanced split graphs on $n+1$ vertices.  
\label{OEIS-thm} 
\end{Thm}

\begin{proof}
 Let $G$ be an $XY$-graph on $n$ vertices and form a split graph $H$ on $n+1$ vertices by adding a new vertex $v$ as follows.  Let $K=Y\cup\{v\}$,  let $S= X$, retain all edges from $G$, and  add an edge between each pair of distinct vertices in $K$.  Since vertex $v$ is not adjacent to any vertex in $S$, we know $S$ is not a maximum stable set in $H$ and thus $H$ is an unbalanced split graph and the given $KS$-partition is $K$-max.

Conversely, let $H$ be an unbalanced split graph on $n+1$ vertices and fix a $KS$-partition of $G$ that is $K$-max.  Chose a swing vertex $v$ in $K$ to remove and let $G$ be the  $XY$-graph with $X = S $, $Y =K-\{v\}$ and $E(G) = \{xy: x \in X, y \in Y, xy \in E(H)\}$.   It is easy to see that these functions are inverses, so they provide a bijection.
\end{proof}

In an $XY$-graph, we say a vertex of $Y$ is an \emph{isolate} if it has no neighbors (in $X$) and that a vertex of $X$ is \emph{universal} if it is adjacent to every vertex of $Y$.

While Theorem~\ref{OEIS-thm} provides a connection between  $XY$-graphs and split graphs,   the number of vertices changes from $n$ to $n+1$ in that result.  In the next theorem, the number of vertices is $n$ for both classes and we see that  $XY$-graphs with no isolates in $Y$ serve as an analog of split graphs. Column 5 in Figure~\ref{big-fig} shows split graphs on four vertices with an $S$-max $KS$-partition; the $XY$-graphs  on four vertices with no isolates in $Y$ are identical  where $X=S$, $Y=K$, and we remove all edges  between vertices in $K$. 

\begin{Thm}  
There is a bijection between   split graphs on $n$ vertices and   $XY$-graphs on $n$ vertices with no isolates in $Y$.   
\label{split-analog-thm}
\end{Thm}

\begin{proof}
Let $G$ be a split graph on $n$ vertices and fix an $S$-max $KS$-partition of $G$.  Let $X =S$ and $Y=K$ and let $H$ be the $XY$-graph with $E(H) = \{xy\in E(G): x \in X, y \in Y\}$.  Then $|V(H)| = n$ and there are no isolates in $Y$ since the $KS$-partition of $G$ was $S$-max.

Conversely, let $H$ be an $XY$-graph on $n$ vertices with no isolates in $Y$.  Form a split graph $G$ by letting $S=X$, $K=Y$ and $E(G) = E(H) \cup \{yz:y,z \in Y, y \neq z\}$.  Then $|V(G)| = n$ and the resulting $KS$-partition is $S$-max since  in graph  $H$ there were no isolates in $Y$.  
These functions are  inverses, so provide a bijection.    
\end{proof}

As we saw in Theorem~\ref{split-analog-thm},  $XY$-graphs with no isolates in $Y$ provide an analog of split graphs.  
%We now seek a compilation theorem for $XY$ graphs with no isolates in $Y$.
We next define balanced and unbalanced for $XY$-graphs with no isolates in $Y$ and then prove that the bijection given in Theorem~\ref{split-analog-thm} preserves balance.  As in our previous definitions of unbalanced, the class is defined by the existence of some structure, in this case, a universal vertex in $X$ (see Table~\ref{unbal-table}).

\begin{Def} {\rm
An $XY$-graph  with no isolates  in $Y$ is \emph{unbalanced} if there   exists a universal vertex in $X$, and \emph{balanced} otherwise.
}
\end{Def}

 Note that we could have instead chosen  $XY$-graphs with a universal vertex in $X$ for our analogue of split graphs.  However, the definition of unbalanced would remain the same, namely, $XY$-graphs with a universal vertex in $X$ and no isolates in $Y$.
 
\begin{Thm}  
The bijection given in the proof of Theorem~\ref{split-analog-thm} maps unbalanced split graphs to unbalanced $XY$-graphs (with no isolates in $Y$) and    balanced split graphs to  balanced $XY$-graphs (with no isolates in $Y$). 
\label{preserve-balance-thm}
\end{Thm}

\begin{proof} By Remark ~\ref{swing-rem}, a split graph is unbalanced if and only if its $S$-max $KS$-partition has a swing vertex in $S$.  An $XY$-graph with no isolate in $Y$ is unbalanced if and only if it has a universal vertex in $X$.  The bijection   given in the proof of Theorem~\ref{split-analog-thm} maps swing vertices in $S$ to universal vertices in $X$ and vice versa, so it preserves balance.
\end{proof}

%Let $G$ be a balanced split graph on $n$ vertices, thus it has a unique $KS$-partition.  Form an $XY$-graph $\ghat$ by letting $X=S$ and $Y=K$ and removing all edges between vertices in $K$.  There are no isolates in $Y$ because the $KS$-partition of $G$ is $S$-max and no universal vertices in $X$ because it is also $K$-max.

%Conversely, let $\ghat$ be an $XY$-graph with no isolated vertices in $Y$ and no universal vertices in $X$.  Let $S=X$ and $K=Y$ and aden an edge between each pair of vertices in $K$.  The result isa balanced split graph since the $KS$-partition is both $K$-max and $S$-max.

%These functions are clearly inverses of one another, and thus each is a bijection.

The next result is our Compilation Theorem  for $XY$-graphs with no isolates in $Y$.      Note that by definition, an $XY$-graph that is balanced or unbalanced has no isolates in $Y$.

\begin{Thm}
There is a bijection between the set of unbalanced $XY$-graphs on $n$ vertices  and the union of the sets of  unbalanced and balanced $XY$-graphs on $t$   vertices where $0 \le t \le n-1$. %that have no isolates in $Y$ .
\label{xy-compilation-thm}
\end{Thm}

\begin{proof}
Let $G$ be an unbalanced $XY$-graph  on $n$ vertices  and let $u$ be a universal vertex in $X$.  Form the $X'Y'$-graph $H$ as follows:  let $X' = X - \{u\}$, let $Y'$ consist of the elements of $Y$ that are adjacent  in $G$ to a vertex of $X'$, and $E(H) = E(G) \cap \{xy: x \in X', y \in Y'\}$.    By construction, there are no isolates in $Y'$ and $|V(H)| \le n-1$ as desired.

Conversely, let $H$ be an $X'Y'$-graph on at most $n-1$ vertices  that is balanced or unbalanced,  and hence has   no isolates in $Y'$.   Let $Y$ consist of the vertices of $Y'$ plus   $(n-1) - |V(H)|$  additional isolates  and let $X$ consist of the vertices of $X'$ plus a universal vertex $u$ so that $E(G) = E(H) \cup \{uy: y \in Y\}$.    The result is an $XY$-graph on $n$ vertices with no isolates in $Y$ (because each vertex of $Y$ is adjacent to $u$).    
These functions are inverses and thus provide a bijection between the two classes.
\end{proof}

%In \cite{ChCoTr16} we gave a compilation theorem for split graphs, finding that the number of unbounded split graphs on $n$ vertices equals the number of split graphs on $n-1$ or fewer vertices.  This also follows from combining three results in this section.  Theorem~\ref{split-analog-thm}
%gives a bijection between split graphs on at most $n-1$ vertices and $XY$-graphs on at most $n-1$ vertices with no isolates in $Y$.    Theorem~\ref{xy-compilation-thm}
%gives a bijection between $XY$-graphs  on at most $n-1$ vertices and with no isolates in $Y$ and unbalanced $XY$-graphs on $n$ vertices.  Finally, Theorem~\ref{preserve-balance-thm}
%gives a bijection between unbalanced $XY$-graphs on $n$ vertices and unbalanced split graphs on $n$ vertices.  Combining these gives the desired compilation theorem for split graphs.

\section{Bipartite Posets }
\label{poset-sec}

A \emph{poset} $P$ consists of a non-empty set $V$ together with a relation $\prec$ on $V$ that is irreflexive, transitive, and therefore antisymmetric.  Two elements (or \emph{points})  $u,v, \in V$ are \emph{comparable} if $u \prec v$ or $v \prec u$, and \emph{incomparable\/} otherwise.   A poset has \emph{height at most one} if there do not exist three elements $x,y,z$ with $x \prec y \prec z$.  Such posets are also known as \emph{bipartite posets}. 
  % A \emph{chain} in a poset is a set of elements $v_1, v_2, v_3, \ldots, v_k$ with $v_1 \prec v_2 \prec v_3 \prec \cdots \prec v_k$.   The \emph{height} of an element $v$  in a poset is the largest number of elements in a chain ending at $v$ and the height of a poset is the number of elements in a longest chain. 
In a  biparite poset,  an element  $v$ has \emph{height 1} if  there is a $u$ with $u \prec v$ and height 0 otherwise.   In this paper, we consider \emph{unlabeled\/} posets, that is, two posets are considered the same if they are isomorphic.   Figure~\ref{big-fig} shows all bipartite posets on four points, where the height 1 points are shown in black.   The examples in the last two columns of this figure can be used to illustrate the proof of the next theorem.

\begin{Thm}
There is a bijection between split graphs on $n$ vertices and  bipartite posets  on $n$ elements.
\label{poset-bij-thm}
\end{Thm}

\begin{proof}  Let $P$ be a bipartite poset.  Let $K$ be the set of height one elements of $P$ and $S$ be the set of height 0 elements of $P$.  Form a split graph $G$ with $KS$-partition as follows:  $V(G) = K \cup S$ and $E(G) = \{uv: u \prec v \hbox{ in $P$} \} \cup \{k\ell: k, \ell \in K, k \neq \ell\}$.  
 %by making each comparability of $P$ an edge of $G$ and also adding edges between each pair of vertices in $K$.    The result is a $KS$-partition of $G$.  Since 
 Each  vertex   of $K$  corresponds to a height 1 element of $P$ which by definition must be  comparable to a height 0 element of $P$.    Thus each vertex of $K$ is  adjacent to a vertex of $S$ and consequently, the $KS$-partition of $G$ is $S$-max.

 Conversely, let $G$ be a split graph on $n$ vertices and fix an $S$-max $KS$-partition of it.     Form a poset  $P = (V, \prec)$ as follows:  $V = K \cup S$ and $x \prec y$ if and only if $x \in S, y \in K$, and $xy \in E(G)$.   %by removing all edges between elements of $K$ and making each edge of the form $ks$ (where $k \in K$ and $s \in X$) into the comparability $k \prec s$.    
 Since the $KS$-partition of $G$  is $S$-max, each element of $K$ is adjacent to an element of $S$, and thus the elements of $K$ become height one elements of $P$.
 
 These functions are inverses and thus provide bijections.  
\end{proof}

We next introduce the notions of balanced and unbalanced for bipartite posets.  Once again, unbalanced bipartite posets are defined by the existence of some structure, in this case, a   full support point (see Table~\ref{unbal-table}).    In column 4 of  Figure~\ref{big-fig}, the full support points are shown in the lighter shade of gray.
 %posets in the first eight rows are unbalanced and the poset in the last row is balanced.  The shaded points are those of full support.
  
\begin{Def}{\rm
 Let $P$ be a bipartite poset.  A point at  height 0 is called a \emph{full support point} if it is comparable to every height one point.   Points at height 0  that are not full support points are called \emph{partial support points}. Poset $P$  is \emph{unbalanced} if it contains a     full support  point and \emph{balanced} otherwise.}
\end{Def}

\begin{Thm} The bijection  in Theorem~\ref{poset-bij-thm}
 preserves balance.

\end{Thm}

\begin{proof}
Let $P$ be an unbalanced  bipartite poset   and let $v$ be a full support  point in $P$.  We form a split graph $G$ as  in the proof of Theorem~\ref{poset-bij-thm}.   Since $v$ is a full support point in $P$, it is a vertex in $S$ that in $G$  is adjacent to all vertices in $K$.  Thus our $KS$-partition of $G$ is not $K$-max and consequently, $G$ is an unbalanced split graph.

Now suppose $P$ is a balanced  bipartite poset  and consider the resulting $KS$-partition of $G$.  The partition is $K$-max since there is no point in $P$ which is a full support point.  It is $S$-max since each height one point in $P$ is comparable to some height 0 point in $P$.  Thus $G$ is a  balanced split graph.
\end{proof}

The next theorem is our Compilation Theorem for Bipartite Posets.  While it follows from previous results,  the proof we give here is a direct proof about posets.

\begin{Thm}
There is a bijection between the set of unbalanced, bipartite posets on  $n$ points  and the set of bipartite posets on at most $t$ points for $0 \le t \le n-1$.
\end{Thm}

\begin{proof}
Let $P$ be an $n$-element unbalanced, bipartite poset.  We define $f(P)$ as follows.  If each height one point in $P$ is comparable to a  partial support point, then simply remove all full support points to arrive at $f(P)$. Note that $f(P)$ has no full support points. Otherwise, there is a height one point in $P$ 
%are interchangeable and each 
that is comparable to precisely the set of full support points.  In this case, choose one such point, say  $u$.  Form $f(P)$  by removing all full support points and making $u \prec v$  for all $v \neq u$ that are height 1 points of $P$.  By construction, the full support points of $f(P)$ include $u$.

Next we define the inverse function $g$.  Let $Q$ be a bipartite poset with $t$ points where $t \le n-1$.  If $Q$ has no full support point, add $n-t$ points at height 0 to $g(Q)$ and make them comparable to all height one points.  The resulting $n$-point bipartite poset is $g(Q)$ and by construction,  the new points are exactly the full support points of $g(Q)$. 
Otherwise,  there is at least one full support point in $Q$.
%they are all interchangeable in the sense that the elements are unlabeled and the points are comparable to the same set of elements.  
In this case, 
choose one such point $v$ to become a height 1 element, and 
%and make $u \prec v$ in $g(Q)$ for all height 0 vertices $u$ in $Q$ other than $v$.  
add $n-t$ points at height 0  and make them comparable to all height one points.   The resulting $n$-point bipartite poset is $g(Q)$ and  again, the new points are exactly the full support points of $g(Q)$. Since $n-t\geq 1$, in either case the resulting poset must have a full support point. Hence the result is an unbalanced, bipartite  poset on $n$-points.

We must show that for  all $Q$, $f(g(Q)) = Q$ and likewise for all  $P$, $g(f(P) )= P$. 
Start with a bipartite poset $Q$  with $t$  points where $0 \le t \le n-1$.    First consider the case in which there are no full support points in $Q$.    In this case, function $g$ adds a set of $n-t$ vertices at height 0 and makes them each a full support point to result  in $g(Q)$.    These $n-t$ elements are the only full support points in $g(Q)$.  Now in $g(Q)$, each height 1 point   is comparable to a partial support point, for otherwise it would not be at height 1 in $Q$.   When $f$ is applied to $g(Q)$, it removes the $n-t$ full support points and no other changes are made, so $f(g(Q)) = Q$.

Next consider the case in which there is a full support point in $Q$.  Such vertices all have the same comparabilities and are thus interchangeable. The function $g$ chooses one such vertex $v$, moves it to height 1 and adds a set $B$ of $n-t$ full support vertices to arrive at $g(Q)$.  In $g(Q)$, the set of full support vertices is exactly $B$.
 %so when $f$ is applied to $g(Q)$, they are removed.  
 Vertex $v$ is   interchangeable with any other point at height 1 in $g(Q)$ whose only comparabilities are to points in $B$.  When $f$ is applied to  $g(Q)$,  the elements of $B$ are removed and $v$  becomes a height 0 point  which is comparable to all remaining points at height 1.  Thus, $f(g(Q))=Q$.

Finally, we show $g(f(P) )= P$.     Start with an unbalanced, bipartite poset $P$ on $n$ points.        First consider the case in which all height 1 points are comparable to a partial support point.  In this case, function $f$  removes all  full support points to arrive at $f(P)$.  Since $P$ is unbalanced, $f$ removes at least one point. Note that there are no full support points in $f(P)$ and all points at height 1 in $f(P)$ are also at height 1 in $P$.  Now apply $g$ to $f(P)$.  This adds the same number of full support points as were removed by $f$, resulting in the original poset $P$.

Now consider the case in which there are points at height 1 whose only comparabilities in $P$ are to the full support points.  These points are interchangeable.  When $f$ is applied to $P$, one such point $v$ is chosen and moved to height 0 in $f(P)$  where it becomes a full support point in $f(P)$ and the original full support points in $P$ are removed. In this case,  $f(P)$ has at least one full support point, namely $v$, and if it has others, they are interchangeable with $v$.  When $g$ is applied to $f(P)$,  the same number of full support points are added as were removed by $f$, and point $v$ is moved to height 1 where it is above the new points, resulting in the original poset. Thus $g(f(P))= P$.
\end{proof}

\section{Remaining Bijections}

In each of the previous three  sections, we gave a bijection  between split graphs and another combinatorial class.  Here we compare these combinatorial classes to each other, giving a direct bijection between each pair of the three new classes and proving these bijections preserve balance.  The bijection in Theorem~\ref{set-cov-poset} is
  illustrated in Figure~\ref{big-fig} by comparing entries in the same row of columns 3 and 4.

\begin{Thm} There is a balance-preserving bijection between minimal set covers on a set with $n$ vertices  and  bipartite  posets with $n$ elements.   
\label{set-cov-poset}
\end{Thm} 
\begin{proof}
 Let $ \cal C$ be a minimal set cover on an $n$-set $V$ and form a bipartite poset $P$ as follows. Choose one (representative)   loyal element from each set in $\cal C$ and let $S$ be the collection of these loyal elements.  Form $P = (V, \prec)$ by placing the elements of $S$ at height 0 and making $x \prec y$ precisely when $x \in S$, $y \in V-S$, and $x$ and $y$ are together in some set of $\cal C$.  Since each $y \in V-S$ is in some set of $\cal C$ and that set has a loyal element in $S$, each element of $V-S$ is at height 1.

The process is reversible.   Let $P$ be a bipartite poset with $n$ elements and form a minimal set cover $\cal C$ as follows.  For each height 0 element in $P$,  form  a set in $\cal C$  containing it and the points it is comparable to in $P$.  These sets form   a cover because each height 1 element is comparable to at least one height 0 element.  By {set-cov-poset}, the  points coming from height 0 elements are loyal, so by Remark~\ref{loyal-rem},
  $\cal C$ is minimal.

Furthermore, note that the number of height 0 elements in $P$  is $|{\cal C}|$, and therefore the number of height 1 elements is $|V| - |{\cal C}|$.  An element at height 0 will be a  full support point of $P$  if and only if  it is the representative loyal vertex chosen from  a set of $\cal C$ with $|V| - |{\cal C}| + 1$ elements.  Thus $\cal C$ is unbalanced precisely when $P$ is unbalanced.

%I THINK WE CAN OMIT THE NEXT PARAGRAPH

%Note that $|{\cal C}|$ is the number of height 0 elements of $P$ and therefore there are $|V| - |{\cal C}|$ height 1 elements of $P$.  A full support vertex of $P$ becomes a loyal vertex in a set of size $|V| - |{\cal C}| + 1$.   Hence $P$ has a full support vertex if and only if $\cal C$ contains a set of size $|V| - |{\cal C}| + 1$.  The only way to get a set of size $|V| - |{\cal C}| + 1$ in $\cal C$ is from a full support vertex.
\end{proof}

\begin{Thm} There is a balance-preserving  bijection between $XY$-graphs  on $n$ vertices with no isolates in  $Y$  and  minimal set covers on a set with $n$ vertices.
\end{Thm}

\begin{proof}
  Let $G $ be an $XY$-graph on $n$ vertices with no isolate in $Y$.  For each vertex $x$ in $X$, form a set consisting of  $x$ and its neighbors. The collection $\cal C$ of these sets is a set cover of $X \cup Y$ because $Y$ has no isolates, so each vertex in $Y$ has a neighbor in $X$.    Each set in $\cal  C$ contains precisely one element of $X$, and that element is loyal, so $\cal  C$ is minimal.

Conversely, let   $\cal C$ be a minimal set cover of a set  $V$ with  $n$ elements. Choose a (representative) loyal vertex from each set in $\cal C$,  let $X$ be the collection of this set  of loyal vertices,  and  let  $Y= V-X$.  Form $XY$-graph $G$ by making $xy \in E(G)$ if and only if $x \in X$, $y \in Y$, and $x$ and $y$ are together in a set of $\cal C$.  There are no isolates in $Y$ because every set in $\cal C$ has a loyal element in $X$.   It is easy to see that these functions are inverses, so they provide a bijection.

By construction,  $|{\cal  C}| = |X|$.  There exists a universal vertex in $X$ if and only if the corresponding set of $\cal C$  has $|Y| + 1$  elements.  Since $|Y| + 1 = |V| - |X| + 1 = |V| - |{\cal C}| + 1$, unbalanced $XY$-graphs map to unbalanced minimal set covers.
\end{proof}

\begin{Thm}   There is a balance-preserving bijection between $XY$-graphs on 
 $n$ vertices with no isolates in $Y$ and bipartite posets on  $n$ elements. 
\end{Thm}

\begin{proof}
   Let $G $ be an $XY$-graph on $n$ vertices with no isolate in $Y$.   Form a bipartite poset $P$ by placing each element of $X$ at height 0 and making $x \prec y$ if and  only if $x\in X$, $y \in Y$ and $xy \in E(G)$.  Since in graph $G$ there are no isolates in $Y$, we know that the elements of $Y$ are at height 1 in bipartite poset $P$.  

Conversely, let $P$ be a height 1 poset with $n$ elements.  Let $X$ be the set of height 0 points and $Y$ be the set of height 1 points.  Form $XY$-graph $G$ by making $xy \in E(G)$ if and only if $X \in X$, $y \in Y$, and $x \prec y$ in $P$.  Since the elements of $Y$ are at height 1 in $P$, there are no isolates in $Y$ in graph $G$.   It is easy to see that these functions are inverses, so they provide a bijection.  

Moreover, a vertex in graph $G$ is a universal vertex of $X$ if and only if it is a full support point in the corresponding bipartite poset $P$.  Thus the bijection given preserves balance.
\end{proof} 

\section{Compilation Theorem for Split Graphs} \label{comp-for-split-g}

In this section we give a direct proof of Theorem~\ref{compil-split-thm}. The inspiration for this proof came from the bijections between split graphs and the other classes considered in this paper. 

\bigskip
\noindent
\emph{Proof of Theorem~\ref{compil-split-thm}.} \ 
Let $G$ be an unbalanced split graph with $n$ vertices and fix an $S$-max $KS$-partition of $G$. By Remark~\ref{swing-rem} there exists a swing vertex $s\in S$. Let $S'=S-\{s\}$ and $K'$ be the set of vertices in $K$ that are adjacent to a vertex in $S'$. Then $K'\cup S'$ is a split graph partition of a graph $H$ with $|V(H)|\leq n-1$. Since each vertex in $K'$ is adjacent to a vertex in $S'$, there are no swing vertices in $K'$ and thus this partition of $H$ is $S$-max. 

Conversely, let $H$ be a split graph with $|V(H)|\leq n-1$. Let $K'S'$ be an $S$-max    partition of $H$. Form graph $G$ with vertex set $K\cup S$ as follows: let $S=S'\cup \{s\}$, where $s$ is a new vertex, and let $K $ consist of the vertices in $K'$ plus enough additional vertices to make $|K|+|S|=n$. The edge set of $G$ consists of the edges of $H$, edges between each pair of distinct vertices in $K$, and an edge between $s$ and each vertex in $K$.  Since $s$ is a swing vertex in this partition, $G$ is unbalanced and the partition is $S$-max. The second mapping reverses the first, and thus we have a bijection. 
\qed

\section{Concluding Remarks}
\label{concl}

In this paper we have presented compilation theorems for split graphs, minimal set covers, $XY$-graphs, and bipartite posets.  We can define compilation theorems in more general settings as follows.  Let $({\cal G})_n$ be a  class of combinatorial objects  on a ground set of size $n$, let  $ ({\cal G})_{\le n-1} = ({\cal G})_0 \cup ({\cal G})_1 \cup ({\cal G})_2 \cup  \cdots \cup ({\cal G})_{n-1}$, and 
and $({\cal H})_n$ be a subclass  of  $({\cal G})_n$ whose members contain an additional natural property.   We say the pair $({\cal G}, {\cal H})$ has a compilation theorem if $|({\cal H})_n| = |({\cal G})_{\le n-1}|$ for each $n \ge 1$.  For example, Theorem~\ref{compil-split-thm} is the compilation theorem for the pair (split graphs, unbalanced split graphs).   The compilation theorems discussed in this paper all generate the sequence $1, 2, 4, 8, 17, 38, 94, 
\ldots$, but other sequences are possible. 
 For example, consider the sequence $1, 2, 4, 8, 16, 32, 64, \ldots$.  Let $({\cal G})_n$ be the set of subsets of an $n$-set and $({\cal H})_n$ be the set of non-empty subsets of an $n$-set.  Then $|({\cal H})_n| = 2^n - 1$  and $|({\cal G})_{\le n}| = 1 + 2 + 4 + \cdots + 2^{n-1} = 2^n-1$, so there is a compilation theorem for the pair $({\cal G}, {\cal H})$.  In future work we hope to find compilation theorems in other contexts.


\begin{thebibliography}{99}  

\bibitem{Be61} C. Berge. F\"arbung von Graphen, deren s\"amtliche bzw. deren ungerade Kreise starr sind, \emph{Wiss. Z. Martin-Luther-Univ. Halle-Wittenberg Math.-Natur. Reihe}, {\bf 10} (1961) 114.

 	
\bibitem{Br75} T. H. Brylawski. An affine representation for transversal geometries, \emph{Studies in Appl. Math.} {\bf 54(2)} (1975) 143-160. 
   
\bibitem{ChCoTr16} C. Cheng, K. L. Collins and A. Trenk.  Split graphs and Nordhaus-Gaddum graphs, \emph{Disc. Math.,} {\bf 339} (2016) 2345-2356.

\bibitem{ChRoSeTh06} M. Chudnovsky, N. Robertson, P. Seymour and R. Thomas. 
The strong perfect graph theorem, 
\emph{Ann. of Math.},  {\bf 21(1)} 164 (2006), 51-229. 

\bibitem{Cl90} R. J. Clarke. Covering a set by subsets, \emph{Disc. Math.,} {\bf 81} (1990) 147-152.

\bibitem{GaGe14} A. Gainer-Dewar and I. M. Gessel.  Enumeration of bipartite graphs and bipartite blocks, \emph{Elect. J. of Combin.}, {\bf 21(2)} (2014) \#P2.40.

\bibitem{Go80}
M. C. Golumbic.
\newblock {\em Algorithmic Graph Theory and Perfect Graphs},
\newblock Academic  Press, New York, 1980.

\bibitem{GuMoRo14} M. Guay-Paquet, A. Morales and E. Rowland. 
Structure and enumeration of $(3+1)$-free posets, \emph{Ann. Comb.}, {\bf 18(4)} (2014),   645-674. 

\bibitem{HaSi81}   P. Hammer and B. Simeone.  The splittance of a graph,  \emph{Combinatorica}, {\bf 1}, 275-284 (1981).

\bibitem{Ha79} P. Hanlon.  The enumeration of bipartite graphs, \emph{Disc. Math.}, {\bf 28} (1979) 49-57. %gen func for connected bipartite graphs

\bibitem{Ro00} Gordon Royle.  Counting set covers and split graphs, \emph{J. of Integer Sequences}, {\bf 3} (2000) Article 00.2.6.
 


\bibitem{We01}   D. B. West.
\newblock {\em Introduction to Graph Theory},
\newblock Prentice Hall, NJ, 2001.


 
  \end{thebibliography}
 \end{document}